\documentclass[12pt]{article}
\usepackage{amsmath,amsthm,amssymb}
\usepackage{hyperref}

\theoremstyle{plain}
\newtheorem{theorem}{Theorem}[section]
\newtheorem{corollary}[theorem]{Corollary}
\newtheorem{lemma}[theorem]{Lemma}
\newtheorem{proposition}[theorem]{Proposition}

\theoremstyle{definition}
\newtheorem{definition}[theorem]{Definition}
\newtheorem{example}[theorem]{Example}
\newtheorem*{runningexample}{Running example}

\theoremstyle{remark}
\newtheorem{remark}[theorem]{Remark}

\newcommand{\F}{\mathbb{F}}
\newcommand{\Q}{\mathbb{Q}}
\newcommand{\Z}{\mathbb{Z}}
\newcommand{\seqnum}[1]{\href{https://oeis.org/#1}{\rm \underline{#1}}}

\title{Integers in Sequences Generated by Recurrences with Non-Integer Coefficients}

\author{Max Lippmann\\
Roslyn High School\\
Roslyn Heights, NY 11577\\
USA\\
\texttt{max.lippmann@yahoo.com}}

\date{}

\begin{document}

\maketitle

\begin{abstract}
In this paper, we study the properties of linear second-order recurrences whose coefficients are non-integer rational numbers. We examine the set of indices at which the terms in the sequence are integers. Specifically, we prove two results. First, after fixing coefficients to satisfy certain coprimality assumptions, we show that there exist coprime integer initial conditions so that the set of integer terms of that sequence is as large as desired. Second, using the same coprimality assumptions, for each fixed pair of initial conditions, we give an explicit cutoff index after which no further terms are integral.
\end{abstract}

\medskip


\medskip


\section{Introduction}\label{sec:introduction}

 Fix integers $a,b,c,d$ with $b,d\ge 2$ and $\gcd(a,b)=\gcd(c,d)=1$, and consider the recurrence
\begin{equation}\label{eq:Sn-rec}
S_{n+2}=\frac{a}{b}S_{n+1}+\frac{c}{d}S_n,\qquad S_0,S_1\in\Z.
\end{equation}

\begin{example}\label{ex:scarcity}
For this example, consider the recurrence: $S_{n+2} = \frac{1}{2}S_{n+1} + \frac{1}{3}S_n$ with $S_0 = 2$ and $S_1 = 7$. This selection of parameters generates the terms:
\[
2,\ 7,\ \frac{25}{6},\ \frac{53}{12},\ \frac{259}{72},\ \frac{157}{48},\ \frac{2449}{864},\ \frac{4333}{1728},\ldots.
\]
After the initial terms, we notice that the denominators of the fractions grow rapidly. Therefore, large cancellations must occur in this sequence to generate integers, which makes them a rare occurrence.
\end{example}

Since integrality should be so rare in these sequences, we are led to ask the following questions. For fixed coefficients $a/b$ and $c/d$, can one choose coprime initial terms $(S_0,S_1)$ so that $S_n\in\Z$ at arbitrarily many distinct indices? Secondly, once the coefficients and coprime initial terms are fixed, can one find an index $N$ such that $S_n\notin\Z$ for all $n\ge N$? The first thing one notices is that one can take advantage of the connection between the non-integer-coefficient recurrence~\eqref{eq:Sn-rec} and the study of an integer sequence. More specifically, we make use of an auxiliary sequence as follows.

Let $K:=bd$. Now we can define the auxiliary sequence like so:
\[
T_n:=K^nS_n\qquad(n\ge 0).
\]
Then $T_n$ satisfies an integer-coefficient recurrence, and $S_n\in\Z$ holds exactly when $K^n$ divides $T_n$.

\begin{proposition}[Auxiliary integer recurrence]\label{prop:aux-recurrence}
If $K$ and $T_n$ are as defined above, then $T_n$ satisfies the second-order linear recurrence with integer coefficients:
\begin{equation}\label{eq:Tn-rec}
T_{n+2}=ad\,T_{n+1}+b^2cd\,T_n, \text{ where } T_0=S_0,\ T_1=KS_1.
\end{equation}
Furthermore,
\begin{equation}\label{eq:divisibility}
S_n\in\Z \text{ if and only if } K^n\mid T_n.
\end{equation}
\end{proposition}

\begin{proof}
Multiply \eqref{eq:Sn-rec} by $K^{n+2}$ and use $K=bd$ together with the definition $T_m=K^mS_m$ to obtain \eqref{eq:Tn-rec}. Since $S_n=T_n/K^n$, \eqref{eq:divisibility} follows.
\end{proof}
\begin{example}\label{ex:oeis-a133594}
To demonstrate this paper's relevance to the study of integer sequences, we look at an example from the OEIS. Specifically, taking
\[
a=1,\qquad b=2,\qquad c=1,\qquad d=3,
\]
we obtain from Proposition~\ref{prop:aux-recurrence} that \(K=bd=6\), and the auxiliary sequence satisfies
\[
T_{n+2}=3T_{n+1}+12T_n.
\]
This integer recurrence appears in the OEIS: the sequence \seqnum{A133594} satisfies
\[
A_{n+2}=3A_{n+1}+12A_n,\qquad A_1=3,\quad A_2=18.
\]
Thus a shifted and scaled copy of \seqnum{A133594} gives the auxiliary sequence for the fractional recurrence
\[
S_{n+2}=\frac12S_{n+1}+\frac13S_n.
\]
Specifically, by setting $T_n=\frac{A_{n+1}}{3}$ we have that
\[
T_0=1,\qquad T_1=6,\qquad T_2=30,\ldots,
\]
and therefore $S_n=\frac{T_n}{6^n}$ starts with the terms
\[
S_0=1,\qquad S_1=1,\qquad S_2=\frac56,\ldots.
\]
This gives a concrete example from the OEIS of an auxiliary integer recurrence associated to a fractional-coefficient recurrence.
\end{example}
\begin{remark}\label{rem:normalizations}
Here we introduce a few important assumptions we will use throughout the paper.
\smallskip\noindent

\textit{(i) Coprime initial terms.}
Coprime initial terms is a very important assumption since it removes many trivial cases where early integrality can be easily generated through large common factors of the initial two terms in the sequence. To illustrate write \(S_0=gx\) and \(S_1=gy\) with \(g\in\Z_{>0}\) and with $g=\gcd(gx,gy)$. Because the recurrence we are working with is linear, each term must be $g$ times the corresponding term of the sequence with initial terms $x$ and $y$. Therefore, because $S_n$ always contains an extra multiple of $g$, if $g$ is divisible by high powers of \(bd\), we can artificially generate a long initial stretch of integral terms. This method of integer generation is not interesting to study, however, since it makes early integrality behavior mostly depend on the choice of $g$ rather than the interesting cancellation patterns of the recurrence itself.
For a more concrete example, take
\[
S_{n+2}=\frac12S_{n+1}+\frac13S_n.
\]
If we choose non-coprime initial terms
\[
S_0=216,\qquad S_1=216,
\]
we find that the first five terms of the sequence are all integers:
\[
216,\ 216,\ 180,\ 162,\ 141,\ldots,
\]
However, this just comes from the fact that the initial terms have the large common factor
\[
216=6^3.
\]
To demonstrate why that is the case, we divide the initial terms by this common factor and instead start with the coprime initial terms
\[
\widetilde S_0=1,\qquad \widetilde S_1=1,
\]
then the surprising integrality in the sequence disappears:
\[
1,\ 1,\ \frac56,\ \frac34,\ \frac{47}{72},\ldots.
\]
Therefore, without the condition $\gcd(S_0,S_1)=1$, we lose sight of the real question concerning the cancellation properties of these sequences.

\smallskip\noindent

\textit{(ii) Lowest terms and separated denominators.}
These assumptions are a bit less intuitive than the assumption involving coprime initial conditions. To illustrate their necessity, we write the rational coefficients in lowest terms, $\frac{a}{b},\ \frac{c}{d}\in\Q,\space b,d\ge2,\space \gcd(a,b)=\gcd(c,d)=1,$ so every prime $p\mid d$ also does not divide $c$.
Additionally we assume $\gcd(b,d)=\gcd(a,d)=1,$ so the prime factors of $b$ and $d$ are disjoint and, for each $p\mid d$, we also have that $p\nmid ab$. This assumption also rules out some degenerate examples in which integrality comes from a characteristic polynomial factoring rather than from the intriguing divisibility behavior studied in this paper. This recurrence serves as a warning as to what can happen when an unwanted factorization happens in a characteristic polynomial.
\[
S_{n+2}=\frac52 S_{n+1}+\frac32 S_n,\qquad S_0=1,\quad S_1=3.
\]
The sequence begins
\[
1,\ 3,\ 9,\ 27,\ 81,\ldots,
\]
We quickly see that this sequence is just the powers of $3$, listed in the OEIS as \seqnum{A000244}. Thus every term is integral. However, this happens for a rather anticlimactic reason: the characteristic polynomial factors as
\[
X^2-\frac52 X-\frac32=(X-3)(X+\frac12).
\]
The initial conditions eliminate the $(-\frac12)^n$ component, leaving only $S_n=3^n$. Thus the integrality here is not caused by cancellation in the recurrence; it is just caused by the sequence collapsing to the integer first-order recurrence $S_{n+1}=3S_n$. This is the kind of example excluded by the assumption $\gcd(b,d)=1$. These assumptions are used repeatedly in certain arguments in Sections~\ref{sec:prescribed} and~\ref{sec:cutoff}.
\end{remark}

The next example shows that, even though in these sequences, integers are hard to come by, as shown in Example~\ref{ex:scarcity}, certain coprime initial terms can still result in integer terms at later indices.

\begin{example}\label{ex:later-integers}
Here we examine the recurrence:
\[
S_{n+2} = \frac{1}{3}S_{n+1} + \frac{1}{2}S_n,
\qquad
S_0 = 1,\quad S_1 = 4683.
\]
The terms of the recurrence go as follows
\[
\begin{gathered}
S_0=1,\quad S_1=4683,\quad S_2=\tfrac{3123}{2},\quad S_3=2862,\\
S_4=\tfrac{6939}{4},\quad S_5=\tfrac{8037}{4},\quad S_6=\tfrac{12297}{8},\quad S_7=1517,\ldots
\end{gathered}
\]
Now we see that $S_3,S_7\in\Z$, while the other listed terms with index greater than one are non-integral.
\end{example}

Corvaja and Zannier \cite{CorvajaZannier2002} proved that, under our coprimality assumptions, a recurrence of the form \eqref{eq:Sn-rec} cannot have $S_n\in\Z$ for infinitely many indices $n$. In a similar vein, Delanoy's Math StackExchange discussion observes that rational recurrences with non-integral coefficients face strong denominator obstructions, making integer-valued examples difficult to obtain except in special circumstances \cite{DelanoyMSE}. Although these results and observations do answer some natural questions about these recurrences, they inspire us to ask the following related question: can one choose coprime initial terms so that integrality occurs for every term in an arbitrarily long but finite list of indices?

The rest of this paper answers both of these questions, each with its own theorem. The theorem answering the first question shows that, for fixed coefficients, one can force arbitrarily many indices in a given recurrence by carefully constructing special coprime initial terms. The theorem answering the second question finds, for a given recurrence with fixed coefficients and fixed coprime initial terms, an index number, such that past that index number, integrality becomes impossible.
\begin{theorem}\label{thm:main-prescribe}
Let $a,b,c,d\in\mathbb Z$ with $b,d\ge 2$. Assume
\[
\gcd(a,b)=\gcd(c,d)=1,\qquad \gcd(a,d)=\gcd(b,d)=1.
\]
For every $k\ge 1$, there exist coprime integers $S_0,S_1$ and a distinct set of indices $2\le n_1<\cdots<n_k$ so that in the recurrence
\[
S_{n+2}=\frac ab S_{n+1}+\frac cd S_n
\]
for each $j$ we have $S_{n_j}\in \mathbb{Z}.$
\end{theorem}

Theorem~\ref{thm:main-prescribe} is proven first, in Section~\ref{sec:prescribed}. For the second theorem, which we prove in Section~\ref{sec:cutoff}, we use $p$-adic valuations, introduced here: For a prime $p$, we write $\nu_p(z)$ for the usual $p$-adic valuation \cite{Koblitz1984}. Note the following property of the $p$-adic valuation of the sum of two numbers:
\[
\nu_p(X+Y)\ge \min(\nu_p(X),\nu_p(Y)),
\]
with equality whenever $\nu_p(X)\ne \nu_p(Y)$. We use this property often in various parts of Section~\ref{sec:cutoff}.

\begin{theorem}\label{thm:main-cutoff}
Let $a,b,c,d,S_0,S_1\in\mathbb Z$ with $b,d\ge 2$ and
\[
\gcd(a,b)=\gcd(c,d)=\gcd(S_0,S_1)=1,\qquad \gcd(a,d)=\gcd(b,d)=1.
\]
Let $S_n$ be defined by the recurrence relation:
\[
S_{n+2} = \frac{a}{b} S_{n+1} + \frac{c}{d} S_n.
\]
Then there exists an index $k$ (depending on $a,b,c,d,S_0,S_1$) such that $S_n\notin\mathbb Z$ for all $n\ge k$.
\end{theorem}

\section{Integers at desired indices}\label{sec:prescribed}

In this section we prove Theorem~\ref{thm:main-prescribe}. We modify the integrality condition on $S_n$ into an equivalent question relating $p$-adic valuation inequalities for a new integer valued linear recurrence called $T_n$. We then separate the primes we are tracking the $p$-adic valuation of into two classes: regular primes and special primes. After finding a method to construct compatible coprime initial conditions such that the $p$-adic inequalities are satisfied separately for each class, we glue the initial conditions together using the Chinese remainder theorem.

\subsection{Converting to the Auxiliary Sequence}

Throughout this section we use the assumptions on $a,b,c,d\in\Z$ with $b,d\ge 2$ and $\gcd(a,b)=\gcd(c,d)=\gcd(a,d)=\gcd(b,d)=1$. Using these assumptions, we look at the recurrence:
\begin{equation}\label{eq:Sn-recurrence}
S_{n+2}=\frac{a}{b}S_{n+1}+\frac{c}{d}S_n,\qquad
S_0,S_1\in\Z,\ \gcd(S_0,S_1)=1.
\end{equation}
Now let $T_n$ be an integer sequence, defined by $K:=bd$ and $T_n:=K^n S_n$. Then $S_n$ being an integer is equivalent to $\nu_p(T_n)\ge n\nu_p(K)$ for all primes $p\mid K$ by Proposition~\ref{prop:aux-recurrence}, we call this newly defined recurrence the auxiliary sequence of $S_n$. Furthermore from this proposition, if $P:=ad$ and $Q:=-b^{2}cd$, then the auxiliary sequence satisfies
\begin{equation}\label{eq:Trec}
T_{n+2}=PT_{n+1}-QT_n,\qquad T_0=S_0,\quad T_1=K S_1.
\end{equation}
For the remainder of the proof, we will work with the auxiliary sequence $T_n$.

\begin{runningexample}
We demonstrate how we prove the theorem by examining the recurrence
\[
  S_{n+2}=\frac13 S_{n+1}+\frac12 S_n,\qquad S_0=x,\ S_1=y,
\]
Our goal is to generate coprime initial conditions $(x,y)$ such that the sequence has integer terms at the indices $3$ and $7$. Our first step is to convert the original sequence with fractional coefficients into its auxiliary integer recurrence. In this example, $a=c=1$, $b=3$, and $d=2$, so
\[
K=bd=6,\qquad P:=ad=2,\qquad Q:=-b^2cd=-18.
\]
Since $K=6$ and $T_n=6^nS_n$, Proposition~\ref{prop:aux-recurrence} alters the recurrence into:
\[
T_0=x,\qquad T_1=6y,\qquad T_{n+2}=ad\,T_{n+1}+b^2cd\,T_n=2T_{n+1}+18T_n.
\]
As proven, $S_n\in\mathbb Z$ exactly when $6^n\mid T_n$.
\end{runningexample}

Now we introduce a new linear second-order recurrence to use for the proof:  let $U_n$ be the Lucas sequence of the first kind for $(P,Q)$:
\begin{equation} \label{eq:lucas}
U_0=0,\quad U_1=1,\quad U_{n+2}=PU_{n+1}-QU_n.
\end{equation}

\begin{lemma}\label{lem:companion}
With $K=bd$, $T_n:=K^n S_n$, and $n>0$, we have
\[
T_n = K S_1U_n - QS_0U_{n-1},
\]
where $U_n$ is the Lucas sequence defined in \eqref{eq:lucas}.
\end{lemma}
\begin{proof}
Set
\[
W_n:=KS_1U_n-QS_0U_{n-1}\qquad(n\ge 1).
\]
Then
\[
W_1=KS_1=T_1.
\]
Since $U_2=P$, we also have
\[
W_2=KS_1U_2-QS_0U_1=PKS_1-QS_0=T_2.
\]
This shows that $W_n$ and $T_n$ have the same initial values for $n=1$ and $n=2$. Also, because $U_{n+2}=PU_{n+1}-QU_n$,
\begin{align*}
W_{n+2}
&=KS_1U_{n+2}-QS_0U_{n+1}\\
&=KS_1(PU_{n+1}-QU_n)-QS_0U_{n+1}\\
&=P(KS_1U_{n+1}-QS_0U_n)-Q(KS_1U_n-QS_0U_{n-1})\\
&=PW_{n+1}-QW_n.
\end{align*}
Since the two sequences $W_n$ and $T_n$ satisfy the same second-order recurrence and agree at $n=1$ and $n=2$, we have $W_n=T_n$ for all $n>0$.
\end{proof}
\begin{runningexample}[Companion identity in this case]
Here, we see the associated Lucas sequence $U_n$ for the running example:
\[
U_0=0,\quad U_1=1,\quad U_{n+2}=2U_{n+1}+18U_n.
\]
We use Lemma~\ref{lem:companion} to find the expression for $T_n$
\[
T_n=KS_1U_n-QS_0U_{n-1}=6y\,U_n+18x\,U_{n-1}.
\]
Proposition~\ref{prop:aux-recurrence}, shows that for each $n\ge 0$: $S_n\in\Z \text{ if and only if } K^n\mid T_n$. An equivalent and more useful form of this is:
\[
S_n\in\Z \text{ if and only if } \nu_p(T_n)\ge n\nu_p(K)\qquad\text{for every prime }p\mid K.
\]
Thus $S_n$ being an integer is equivalent to a set of $p$-adic inequalities for primes dividing $K$. Since $K=2\cdot 3$, $S_n\in\mathbb Z$ is
equivalent to the pair of inequalities
\[
\nu_2(T_n)\ge n\quad\text{and}\quad \nu_3(T_n)\ge n.
\]
Therefore to select proper initial conditions that force $S_3,S_7\in\mathbb Z$, it is enough to demonstrate that using those initial conditions will result in the inequalities being satisfied.
\[
\nu_2(T_3)\ge 3,\ \nu_3(T_3)\ge 3,\qquad
\nu_2(T_7)\ge 7,\ \nu_3(T_7)\ge 7.
\]
\end{runningexample}
\begin{definition}
Let $p\mid K=bd$. We call $p$ \emph{special} if $p\mid d$, and \emph{regular} if $p\mid b$.
Equivalently, under the standing assumptions stated in Remark~\ref{rem:normalizations},  since $\gcd(b,d)=1$, a prime divisor $p$ of $K$
is special exactly when $p\mid P$ and $p\mid Q$, and regular exactly when $p\mid Q$ but $p\nmid P$.
\end{definition}

The above definition allows us to outline the remainder of this proof. First, we determine the conditions on the initial terms needed for divisibility by regular and special primes separately, and then combine them using a slightly modified version of the Chinese remainder theorem.

\begin{runningexample}[Prime split for this example]
Here the primes dividing $K=6$ are $2$ and $3$.
Since $2\mid d$ and $2$ divides both $P=2$ and $Q=-18$, it is special.
Since $3\mid b$ but $3\nmid P=2$, the prime $3$ is regular.
To make the proof simpler, we will deal with divisibility by $2$ separately from divisibility by $3$.
\end{runningexample}

Here we start by dealing with the regular primes $p\mid b$.
\begin{proposition}\label{prop:B-global-glue}
Let $(S_n)$ be a sequence satisfying the equation described in \eqref{eq:Sn-recurrence}, let $K=bd$, and finally let $\mathcal N=\{n_1<\cdots<n_k\}$ be a finite increasing set of positive integers. There must always exist coprime integers
$S_0,S_1$ such that $b^n\mid T_n
\text{ for every }n\in\mathcal N$.
\end{proposition}

\begin{proof}
For notational simplicity, we write $S_0=x$ and $S_1=y$. Now fix a prime $p\mid b$. Since $\gcd(a,b)=1$ and $\gcd(b,d)=1$, we have 
$p\nmid a$ and $p\nmid d$, so $p\nmid P=ad$. Also $Q=-b^2cd$, so $p\mid Q$. Reducing the recurrence $U_{n+2}=PU_{n+1}-QU_n$ modulo $p$ therefore gives
\[
U_{n+2}\equiv P\,U_{n+1}\pmod p.
\]
Using the above fact, it is now very simple to prove by induction that
\[
U_n\equiv P^{n-1}\pmod p
\qquad(n\ge 1).
\]
The $n=1$ case is immediately verified since $U_1=1=P^0$.
Now to deal with the general case: if $U_n\equiv P^{n-1}\pmod p$, then
\[
U_{n+1}\equiv P\,U_n\equiv P\cdot P^{n-1}=P^n\pmod p,
\]
this completes the induction, implying
\[
U_n\equiv P^{n-1}\pmod p
\qquad(n\ge 1).
\]
Since $p\nmid P$, we have $P^{n-1}\not\equiv 0\pmod p$, and therefore
\[
U_n\not\equiv 0\pmod p
\qquad(n\ge 1).
\]
Now we know that for all $n\geq 1$
\begin{equation}\label{eq:Un-unit}
\gcd(U_n,b)=1.
\end{equation}

Now we can use Lemma~\ref{lem:companion},
\[
T_n=KS_1U_n-QS_0U_{n-1}
=bd\,yU_n+b^2cd\,xU_{n-1}
=bd\bigl(yU_n+bc\,xU_{n-1}\bigr).
\]
Since $\gcd(b,d)=1$, the condition that $b^n\mid T_n$ is equivalent to the equality:
\begin{equation}\label{eq:Bn-1}
yU_n+bc\,xU_{n-1}\equiv 0\pmod{b^{\,n-1}}.
\end{equation}

By \eqref{eq:Un-unit}, we have $\gcd(U_n,b)=1$, so $U_n$ has an inverse modulo $b^{n-1}$.
Hence we can isolate $y$ in \eqref{eq:Bn-1} as follows
\[
y\equiv -bc\,x\,U_{n-1}U_n^{-1}\pmod{b^{n-1}},
\]
where $U_n^{-1}$ is the inverse of $U_n$ modulo $b^{n-1}$.
To simplify notation we can set
\[
\rho_n:=-bc\,U_{n-1}U_n^{-1}\pmod{b^{n-1}}.
\]
Then \eqref{eq:Bn-1} simplifies to
\[
y\equiv x\,\rho_n\pmod{b^{n-1}}.
\]
This does look like it could just be a sufficient and necessary condition to satisfy the $p$-adic inequalities for regular primes: just select $x$ and $y$ to be a constant multiple apart modulo $b^{n-1}$. However, we run into an obstacle using this logic, if for some $i,j$ we have that $\rho_i\neq\rho_j$, since $x$ and $y$ cannot change, certain indices with different $\rho$ values could wind up not being able to satisfy the $p$-adic inequalities for regular primes. To ensure that this will not happen, we need to show that if $m\ge n\ge 1$, then
\begin{equation}\label{eq:rho-compat}
\rho_m\equiv \rho_n\pmod{b^{\,n-1}}.
\end{equation}
To prove this we use an unconventional tool, specifically Cassini's identity for Lucas sequences, which states
\[
U_{t+1}U_{t-1}-U_t^2=-Q^{t-1}\qquad(t\ge 1).
\]
From this
\[
\rho_{t+1}-\rho_t
=-bc\left(\frac{U_t}{U_{t+1}}-\frac{U_{t-1}}{U_t}\right)
=-bc\cdot\frac{U_t^2-U_{t+1}U_{t-1}}{U_{t+1}U_t}
=-bc\cdot\frac{Q^{t-1}}{U_{t+1}U_t}.
\]
Since $Q=-b^2cd$ is divisible by $b^2$ and $U_tU_{t+1}$ is relatively prime to $b$ because of
\eqref{eq:Un-unit}, we have the equation:
\[
\rho_{t+1}-\rho_t
\equiv 0\pmod{b^{\,2t-1}}.
\]
This implies that, for every $t\ge n$,
\[
\rho_{t+1}-\rho_t\equiv 0\pmod{b^{\,n-1}}.
\]
Now summing this difference for all possible values of $t$ we learn
\[
\rho_m-\rho_n
=
\sum_{t=n}^{m-1}(\rho_{t+1}-\rho_t)
\equiv 0\pmod{b^{\,n-1}},
\]
since each summand is congruent to \(0\) modulo \(b^{\,n-1}\), rearranging the equation will prove \eqref{eq:rho-compat}. Now we actually choose initial conditions such that the set of regular prime $p$-adic inequalities is solved for the desired set of indices $\mathcal N=\{n_1<\cdots<n_k\}$. First set $n_*:=n_k$, let $x=1$ and choose
$y$ such that
\[
y\equiv \rho_{n_*}\pmod{b^{\,n_*-1}}.
\]
Then for each $n_i\in\mathcal N$, we learn from \eqref{eq:rho-compat} that
\[
\rho_{n_*}\equiv \rho_{n_i}\pmod{b^{\,n_i-1}},
\]
finally achieving that
\[
y\equiv \rho_{n_i}\pmod{b^{\,n_i-1}}.
\]
This implies that \eqref{eq:Bn-1} holds for every $n_i\in\mathcal N$, and gives us the desired divisibility expression
\[
b^{n_i}\mid T_{n_i}
\qquad(1\le i\le k).
\]
Finally, since $x=1$ is coprime to everything, the coprimality of the initial terms is confirmed.
\end{proof}
\begin{runningexample}[Proposition~\ref{prop:B-global-glue} (regular prime $p=3$)]
In our running example we have that $b=3$ and $c=1$. As described in the procedure we take $x=S_0=1$.  For $(P,Q)=(2,-18)$ the needed terms of the Lucas sequence are
\[
U_2=2,\quad U_3=22,\quad U_6=2552,\quad U_7=15112.
\]
To make the valuation at the regular prime $p=3$ large enough at the largest target index $n_\ast=7$, we solve the congruence
\[
yU_7+3U_6\equiv 0\pmod{3^{6}}.
\]
Since $\gcd(U_7,3)=1$, there is a unique solution modulo $3^6$:
\[
y\equiv -3U_6\,U_7^{-1}\pmod{3^6},
\]
and this solution turns out to be $y\equiv 309\pmod{729}$.
For \(n=3\), the congruence is
\[
yU_3+3U_2\equiv 0\pmod{3^2}.
\]
Since \(\gcd(U_3,3)=1\), this is equivalent to
\[
y\equiv -3U_2\,U_3^{-1}\pmod{3^2}.
\]
Since
\[
-3U_2\,U_3^{-1}\equiv 3\pmod{3^2},
\]
we have that
\[
y\equiv 3\pmod 9.
\]
The residue $309$ satisfies this congruence, so as we had hoped single choice of $y$ satisfies both congruences $\nu_3(T_3)\ge 3$ and $\nu_3(T_7)\ge 7$.
\end{runningexample}

Now, to handle the $p$-adic inequalities for special primes, we use a matrix argument. Since matrices keep track of the valuations of the relevant Lucas terms and of their unit parts, we use them to track $p$-adic information about terms in the sequence in ways that the previous arguments could not.

\begin{lemma}\label{lem:special-coprime-clean}
Let $p\mid d$ be a special prime, set $e:=\nu_p(d)$ and
\[
P=p^eP_p,\qquad Q=p^eQ_p,
\]
with $p\nmid P_pQ_p$. Let $m\ge1$ satisfy $p\nmid m$. Then
\[
U_{2m}=p^{em}u_{2m},\qquad U_{2m+1}=p^{em}u_{2m+1},
\]
where $u_{2m},u_{2m+1}\in\Z$ and $p\nmid u_{2m}u_{2m+1}$. Then
\[
r_p(m):=-\frac{u_{2m+1}}{bc\,u_{2m}}
\]
is a $p$-adic unit. Secondly, there is a positive integer $\ell_p$, depending only on $p$ and the recurrence, such that whenever $s,t\ge1$, $p\nmid st$, $N\ge1$, and
\[
s\equiv t\pmod{\ell_p p^N},
\]
$r_p(m)$ has the special property that
\[
r_p(s)\equiv r_p(t)\pmod{p^N}.
\]
\end{lemma}

\begin{proof}
Start by defining:
\[
M=\begin{pmatrix}P&-Q\\1&0\end{pmatrix}.
\]
This encodes the recurrence into matrix form, since
\[
M\begin{pmatrix}U_{n+1}\\U_n\end{pmatrix}
=
\begin{pmatrix}PU_{n+1}-QU_n\\U_{n+1}\end{pmatrix}
=
\begin{pmatrix}U_{n+2}\\U_{n+1}\end{pmatrix}.
\]
Because \(\binom{U_1}{U_0}=\binom{1}{0}\), it follows by induction that
\[
\binom{U_{n+1}}{U_n}=M^n\binom{1}{0}.
\]
Now, we want to separate the even terms in our sequence from the odd terms, and determine their $p$-adic valuation, so we write out:
\[
M^2=
\begin{pmatrix}
P^2-Q & -PQ\\
P & -Q
\end{pmatrix}
=
p^e
\begin{pmatrix}
p^eP_p^2-Q_p & -p^eP_pQ_p\\
P_p & -Q_p
\end{pmatrix},
\]
to simplify our notation we write
\[
A_p:=
\begin{pmatrix}
p^eP_p^2-Q_p & -p^eP_pQ_p\\
P_p & -Q_p
\end{pmatrix}.
\]
It is important to note that this is an \emph{integer matrix}. Then
\[
\binom{U_{2m+1}}{U_{2m}}
=M^{2m}\binom{1}{0}
=p^{em}A_p^m\binom{1}{0}.
\]
Since $A_p$ has integer entries, the vector
\[
A_p^m\binom{1}{0}
\]
has integer coordinates. Thus there exist integers $a_m$ and $b_m$ such that
\[
A_p^m\binom{1}{0}=\binom{a_m}{b_m}.
\]
Therefore
\[
\binom{U_{2m+1}}{U_{2m}}=p^{em}\binom{a_m}{b_m},
\]
so $U_{2m+1}/p^{em}$ and $U_{2m}/p^{em}$ are integers. Modulo $p$, we then have
\[
A_p\equiv
\begin{pmatrix}
-Q_p&0\\
P_p&-Q_p
\end{pmatrix}
\pmod p.
\]
Therefore
\[
A_p^m\binom{1}{0}
\equiv
\binom{(-Q_p)^m}{mP_p(-Q_p)^{m-1}}
\pmod p.
\]
Since $p\nmid mP_pQ_p$, both terms of the matrix are nonzero modulo $p$. This shows that indeed
\[
\nu_p(U_{2m+1})=\nu_p(U_{2m})=em.
\]
Now since we know that the exact $p$-adic valuation of $U_{2n} \text{ and } U_{2n+1}$ is $em$, we can write
\[
U_{2m}=p^{em}u_{2m},\qquad U_{2m+1}=p^{em}u_{2m+1},
\]
with $p\nmid u_{2m}u_{2m+1}$. Since $p\nmid bc$, it is immediate that
\[
r_p(m)=-\frac{u_{2m+1}}{bc\,u_{2m}}\in\Z_p^\times.
\]

Now we prove the second part of the theorem. First note that the determinant of $A_p$ is $Q_p^2$, so $A_p$ is invertible modulo $p$. This implies that the image of $A_p$ in $\operatorname{GL}_2(\F_p)$ is an element with finite order. Now we can choose some positive integer $\ell_p$ such that
\[
A_p^{\ell_p}\equiv I\pmod p.
\]
A concrete example of a suitable $\ell_p$ could be
\[
\ell_p=\#\operatorname{GL}_2(\F_p)=(p^2-1)(p^2-p).
\]
Now we proceed with a direct proof. Assume that $s,t\ge1$, $p\nmid st$, $N\ge1$, and
\[
s\equiv t\pmod{\ell_p p^N}.
\]
Without loss of generality we can also assume that $s\ge t$. From the definition of congruence mod $\ell_p p^N$, we have:
\[
s-t=\ell_p p^N q
\]
for some integer $q\ge0$. To simplify our notation even further, set $B_p:=A_p^{\ell_p}$. Since $B_p\equiv I\pmod p$, we can write, for some matrix $C_p$ over $\Z_p$, $B_p=I+pC_p$. By the binomial expansion of
\[
B_p^{p^Nq}=(I+pC_p)^{p^Nq},
\]
every nonconstant term is divisible by $p^N$. Hence
\[
B_p^{p^Nq}\equiv I\pmod{p^N}.
\]
Therefore
\[
A_p^s=A_p^t\left(A_p^{\ell_p}\right)^{p^Nq}=A_p^tB_p^{p^Nq}\equiv A_p^t\pmod{p^N}.
\]
If we multiply both sides by $\binom{1}{0}$ we achieve
\[
A_p^s\binom{1}{0}\equiv A_p^t\binom{1}{0}\pmod{p^N}.
\]
Now because of the identity
\[
A_p^m\binom{1}{0}=\binom{u_{2m+1}}{u_{2m}},
\]
we are able to obtain that
\[
u_{2s+1}\equiv u_{2t+1}\pmod{p^N},
\qquad
u_{2s}\equiv u_{2t}\pmod{p^N}.
\]
Recalling the first part of the proof, we know that $u_{2s}$ and $u_{2t}$ are $p$-adic units, so their inverses are congruent modulo $p^N$. Therefore
\[
-\frac{u_{2s+1}}{bc\,u_{2s}}
\equiv
-\frac{u_{2t+1}}{bc\,u_{2t}}
\pmod{p^N},
\]
this proves the lemma, since the above statement is exactly equivalent to:
\[
r_p(s)\equiv r_p(t)\pmod{p^N}.\qedhere
\]
\end{proof}

\begin{proposition}\label{prop:single-index}

Let $p\mid d$ be a special prime, and set $\beta:=\nu_p(d)$. Let $m\ge1$ satisfy $\gcd(m,d)=1$. Suppose that $S_n$ is a solution to the recurrence \eqref{eq:Sn-recurrence} with coprime initial terms $S_0,S_1$, both of which are coprime to $d$. Finally, let $K=bd$ and $T_n:=K^nS_n$. Then
\[
\nu_p(T_{2m+1}) \ge (2m+1)\nu_p(K)
\quad\text{if and only if}\quad
\frac{S_0}{S_1} \equiv r_p(m)\mod p^{\beta m}.
\]
\end{proposition}

\begin{proof}
For neatness write $S_0=x$ and $S_1=y$. Now using Lemma~\ref{lem:companion} for the case where $n=2m+1$,
\[
T_{2m+1}=K yU_{2m+1}-QxU_{2m}=bd\,yU_{2m+1}+b^2cd\,xU_{2m}.
\]
Set $\beta:=\nu_p(d)=\nu_p(Q)=\nu_p(K)$, keep in mind that the last equality holds because $p\nmid b$. Since we know that $\nu_p(d)=\beta$, we can write $d=p^\beta d'$ with $p\nmid d'$. Furthermore, because \(p\mid d\) and \(\gcd(xy,d)=1\), we have \(p\nmid x\) and \(p\nmid y\). Also
\(p\nmid b\) and \(p\nmid c\), because \(\gcd(b,d)=\gcd(c,d)=1\). These long list of equivalences allow us to write: 
\[
T_{2m+1}=p^\beta b d'\bigl(yU_{2m+1}+bc\,xU_{2m}\bigr).
\]
Using Lemma~\ref{lem:special-coprime-clean} we can write:
$U_{2m}=p^{\beta m}u_{2m}$ and $U_{2m+1}=p^{\beta m}u_{2m+1}$ with $p\nmid u_{2m}u_{2m+1}$, hence
\[
T_{2m+1}=p^{\beta(m+1)} b d' \bigl(yu_{2m+1}+bc\,x u_{2m}\bigr).
\]
Therefore $\nu_p(T_{2m+1})\ge (2m+1)\beta$ is equivalent to
\[
yu_{2m+1}+bc\,x u_{2m}\equiv 0\pmod{p^{\beta m}}.
\]
Since $p\nmid bc\,y\,u_{2m}$, this is equivalent to
\[
\frac{x}{y}\equiv -\frac{u_{2m+1}}{bc\,u_{2m}} = r_p(m)\pmod{p^{\beta m}}.
\]
This proves the proposition.
\end{proof}

\begin{proposition}\label{prop:special-many-even}
For all $k\geq1$ there exist a set of integers
\[
1\le m_1<\cdots<m_k,\qquad \gcd(m_j,bd)=1\quad (1\le j\le k),
\]
and a residue class defined by
\[
y\equiv y_d \pmod{d^{m_k}},
\]
with \(\gcd(y_d,d)=1\), such that for every integer \(y\) satisfying
\[
y\equiv y_d \pmod{d^{m_k}},
\]
the recurrence \eqref{eq:Sn-recurrence} with initial terms \((S_0,S_1)=(1,y)\)
will satisfy the $p$-adic termwise inequality
\[
\nu_p\!\left(T_{2m_j+1}\right)\ge (2m_j+1)\nu_p(K)
\]
for every special prime \(p\mid d\) and every \(1\le j\le k\).
\end{proposition}

\begin{proof}
Start with an arbitrary $k\ge1$ and set $S_0=1$, $S_1=y$. For a special prime $p\mid d$, Proposition~\ref{prop:single-index} gives
\begin{equation}\label{eq:crit-special}
\nu_p(T_{2m+1}) \ge (2m+1)\nu_p(K)
\quad\text{if and only if}\quad
\frac{1}{y} \equiv r_p(m)\pmod{p^{\nu_p(d)\,m}}.
\end{equation}

For each special prime $p\mid d$, let $\ell_p$ be the integer supplied by Lemma~\ref{lem:special-coprime-clean}, and set
\[
E:=\operatorname{lcm}_{p\mid d}\ell_p.
\]
We will now choose indices $m_1<\cdots<m_k$ so that the inequality on right-hand side of \eqref{eq:crit-special} is true for all chosen indices $m_j$. Start with $m_1:=1$, and suppose $m_j$ is an arbitrary whole number with the property that $\gcd(m_j,bd)=1$. Now choose an $m_{j+1}>m_j$ satisfying
\begin{equation}\label{eq:spacing}
m_{j+1}\equiv m_j \pmod{bE d^{m_j}}.
\end{equation}
Now using \eqref{eq:spacing}:
\[
m_{j+1}\equiv m_j\pmod b
\qquad\text{and}\qquad
m_{j+1}\equiv m_j\pmod d,
\]
showing that every prime dividing $bd$ divides $m_{j+1}$ if and only if it divides $m_j$. Since $\gcd(m_j,bd)=1$, it clearly follows that $\gcd(m_{j+1},bd)=1$. Fix a special prime $p\mid d$ and write $\beta:=\nu_p(d)$. From \eqref{eq:spacing}, and because $E$ is divisible by $\ell_p$, we have
\[
m_{j+1}\equiv m_j\pmod{\ell_p p^{\beta m_j}}.
\]
Lemma~\ref{lem:special-coprime-clean} therefore gives
\[
r_p(m_{j+1})\equiv r_p(m_j)\pmod{p^{\beta m_j}}.
\]
Continuing up inductively allows us to deduce that 
\[
r_p(m_k)\equiv r_p(m_j)\pmod{p^{\beta m_j}}
\qquad\text{for every } j\le k.
\]
This compatibility property is absolutely key, since once $m_j$ has been chosen so that the desired congruence modulo $p^{\beta m_j}$ holds, the same congruence continues to hold for every later $m_k$. Thus the condition for integrality at stage $j$ is the same throughout the rest of the construction of integer terms. Now for each special prime $p\mid d$, write $\beta=\nu_p(d)$. Choose $y$ so that
\[
\frac1y\equiv r_p(m_k)\pmod{p^{\beta m_k}}.
\]
Since $r_p(m_k)\in\Z_p^\times$, this is equivalent to
\[
y\equiv r_p(m_k)^{-1}\pmod{p^{\beta m_k}}.
\]
The set of moduli $p^{\beta m_k}$, as $p$ varies over all the primes dividing $d$, are pairwise coprime, so the Chinese remainder theorem can determine a single residue class
\[
y\equiv y_d\pmod{\prod_{p\mid d}p^{\nu_p(d)m_k}}
=
y_d\pmod{d^{m_k}}.
\]
Furthermore, each residue $r_p(m_k)^{-1}$ is a $p$-adic unit, so it is not divisible by $p$. Therefore the resulting CRT class $y_d$ is not divisible by any prime $p\mid d$, and hence $\gcd(y_d,d)=1$. Now let $y$ be any integer satisfying $y\equiv y_d\pmod{d^{m_k}}$. Then, for each $p\mid d$,
\[
y\equiv r_p(m_k)^{-1}\pmod{p^{\beta m_k}}.
\]
Since both sides are $p$-adic units, we can take their inverse:
\[
\frac1y\equiv r_p(m_k)\pmod{p^{\beta m_k}}.
\]
Reducing this modulo $p^{\beta m_j}$ and then using the property that $r_p(m)$ are compatible:
\[
r_p(m_k)\equiv r_p(m_j)\pmod{p^{\beta m_j}},
\]
we obtain
\[
\frac1y\equiv r_p(m_j)\pmod{p^{\beta m_j}}.
\]
Proposition~\ref{prop:single-index} shows that this is exactly
\[
\nu_p\!\left(T_{2m_j+1}\right)\ge (2m_j+1)\nu_p(K).
\]
This holds for every special prime $p\mid d$ and every $1\le j\le k$.
\end{proof}
\begin{runningexample}[Special prime $p=2$ at the target indices]
To show how the given specific choice for $y$ can satisfy all the desired special prime $p$-adic inequalities at the target indices, we use the example auxiliary recurrence $T_n=6yU_n+18xU_{n-1}$. At the first target index $n=3$ we find that the desired terms in the corresponding Lucas sequence are $U_3=22$ and $U_2=2$, so
\[
T_3=6y\cdot 22+18\cdot 2=132y+36=4(33y+9).
\]
Thus $\nu_2(T_3)\ge 3$ is the same thing as $33y+9\equiv 0\pmod 2$. Next, moving on to the second target index of $n=7$ we have $U_7=15112$ and $U_6=2552$, so
\[
T_7=6y\cdot 15112+18\cdot 2552=90672y+45936=16(5667y+2871).
\]
Thus $\nu_2(T_7)\ge 7$ is equivalent to $5667y+2871\equiv 0\pmod 8$.
Reducing modulo $8$ gives $3y+7\equiv 0\pmod 8$, thus
\[
y\equiv 3\pmod 8.
\]
Furthermore, $y$ being congruent to $3$ mod $8$ also satisfies the first congruence: if $y\equiv 3\pmod 8$, then
\[
33y+9\equiv 33\cdot 3+9=108\equiv 0\pmod 2,
\]
so $T_3=4(33y+9)$ satisfies $\nu_2(T_3)\ge 3$.
\end{runningexample}
We now have two different methods which each generate a unique set of initial conditions where either the special and regular prime $p$-adic inequalities are solved. Now we outline the delicate technique which enables us to glue the local conditions for regular and special primes using the Chinese remainder theorem. This obtains coprime initial terms $(S_0,S_1)$ which force integrality at all desired indices simultaneously. Thus, with a large enough set of desired indices you can construct initial conditions which make any number of desired terms in the recurrence integers.

\begin{proposition}\label{prop:odd-indices}
Let $S_n$ be a sequence that satisfies the recurrence given in \eqref{eq:Sn-recurrence} and the coprimality assumptions of Remark~\ref{rem:normalizations}.
Then, for all $k\geq 1$ there must exist coprime integers $S_0,S_1$ and positive integers
$m_1<\cdots<m_k$, with
\[
\gcd(m_j,bd)=1 \quad (1\le j\le k),
\]
such that $S_{2m_j+1}\in\Z$ for each $1\le j\le k$.
\end{proposition}

\begin{proof}
From Proposition~\ref{prop:special-many-even}, we are allowed to pick integers $m_1<\cdots<m_k$ with $\gcd(m_j,bd)=1$, we can also create a residue class
\[
y\equiv y_d\pmod{d^{m_k}},
\]
with $\gcd(y_d,d)=1$, such that for every integer $y$ in this residue class:
\[
\nu_p\bigl(T_{2m_j+1}\bigr) \ge (2m_j+1)\nu_p(K)\qquad(p\mid d,\ 1\le j\le k).
\]
Now, apply Proposition~\ref{prop:B-global-glue} to the set of indices
$\{2m_1+1<\cdots<2m_k+1\}$ with $S_0=1$ to obtain an integer $y_b$ such that, with $(S_0,S_1)=(1,y_b)$,
\[
\nu_p\bigl(T_{2m_j+1}\bigr) \ge (2m_j+1)\nu_p(b)\qquad(p\mid b,\ 1\le j\le k).
\]
Define $n_k:=2m_k+1$. Propositions~\ref{prop:B-global-glue} and~\ref{prop:special-many-even} determine the congruence classes:
$y\equiv y_b\pmod{b^{n_k-1}}$ and $y\equiv y_d\pmod{d^{m_k}}$.
Because $\gcd(b^{n_k-1},d^{m_k})=1$, the hypotheses of the Chinese remainder theorem are satisfied, and we can produce an integer $y$ satisfying both of the above congruences. Now, set the initial conditions of the sequence to be $(S_0,S_1):=(1,y)$. To verify that indeed, these initial conditions give integers at all the desired spots, we realize that, for all primes \(p\mid b\), we have \(\nu_p(K)=\nu_p(b)\), so Proposition~\ref{prop:B-global-glue} gives
\[
\nu_p(T_{2m_j+1})\ge (2m_j+1)\nu_p(K)
\qquad (1\le j\le k).
\]
In the same vein, for every prime \(p\mid d\) we have \(\nu_p(K)=\nu_p(d)\), so Proposition~\ref{prop:special-many-even} gives
\[
\nu_p(T_{2m_j+1})\ge (2m_j+1)\nu_p(K)
\qquad (1\le j\le k).
\]
Since the chosen integer \(y\) satisfies both congruences: $y\equiv y_b  \pmod{b^{n_k-1}}$
$\text{and }y\equiv y_d \pmod{d^{m_k}},$ the conclusions of Propositions~\ref{prop:B-global-glue} and~\ref{prop:special-many-even} hold simultaneously. Hence, for every prime \(p\mid K\),
\[
\nu_p(T_{2m_j+1})\ge (2m_j+1)\nu_p(K)
\qquad (1\le j\le k).
\]
Therefore, by Proposition~\ref{prop:aux-recurrence}, we have
\[
S_{2m_j+1}\in\Z
\qquad (1\le j\le k).\qedhere
\]
\end{proof}
\begin{runningexample}[CRT gluing across primes (final step)]
From Proposition~\ref{prop:B-global-glue} we obtained the $3$-adic condition
\[
y\equiv 309\pmod{3^6}\quad(\text{i.e., }y\equiv 309\pmod{729}).
\]
Now we calculate the $2$-adic conditions that $y_d$ must satisfy using the method outlined in Proposition~\ref{prop:special-many-even}:
\[
y\equiv 3\pmod 8.
\]
Since $\gcd(8,729)=1$, the Chinese remainder theorem yields a unique class modulo
$8\cdot 729=5832$, namely
\[
y\equiv 4683\pmod{5832}.
\]
Taking $x=1$ and the smallest positive representative for $y$ gives $(S_0,S_1)=(1,4683)$. When these initial conditions are substituted into the recurrence, the terms at the indices $n=3$ and $n=7$ are integers, just as desired.
\[
S_2=\tfrac{3123}{2},\quad S_3=2862,\quad S_4=\tfrac{6939}{4},\quad S_5=\tfrac{8037}{4},\quad S_6=\tfrac{12297}{8},\quad S_7=1517.
\]
\end{runningexample}

\section{Non-integrality and an explicit cutoff}\label{sec:cutoff}

This section describes how we prove Theorem~\ref{thm:main-cutoff}. As a general overview, we start by fixing a prime $p\mid d$. We write $S_n$ in Binet form, namely $\alpha r_1^{n}+\beta r_2^{n}$. Then we factor out the linear form $\Lambda^n+\Gamma^{-1}$. In order for $S_n$ to be integral for increasing values of $n$, the $p$-adic valuation of $\Lambda^n+\Gamma^{-1}$ must grow linearly with $n$. However, Chim~\cite{Chim2025} shows that the $p$-adic valuation of numbers of the form $\Lambda^n+\Gamma^{-1}$ grows at most logarithmically. Clearly, after a certain value of $n$, these bounds contradict each other. We call this value $k$. Thus, finding this mystery value of $k$ yields an exact cutoff beyond which no term can be integral.
\begin{definition}
\label{def:binet-lucas}
Let $r_1,r_2$ be the distinct roots of the polynomial:
\[
R^2-\frac{a}{b}R-\frac{c}{d}=0.
\]
This polynomial is also known as the characteristic polynomial of $S_n$, and its roots have the special property that: every solution, $S_{n+2}$, of the following recurrence: $S_{n+2}=\frac{a}{b} S_{n+1}+\frac{c}{d} S_n$ with initial terms $S_0,S_1$ can be written in the form
\[
S_n=\alpha r_1^{n}+\beta r_2^{n},
\]
where
\[
\alpha=\frac{S_1-r_2 S_0}{r_1-r_2},\qquad
\beta =\frac{r_1 S_0 - S_1}{r_1-r_2}.
\]
Everest, van der Poorten, Shparlinski, and Ward~\cite{EverestVdPShparlinskiWard2003} fully detail this special property of linear second-order recurrences.
\end{definition}

\begin{definition}\label{def:vp-extension}
Fix a prime \(p\mid d\), and set $K_0:=\Q(r_1,r_2)$. Now pick an embedding
\[
\iota_p:K_0\hookrightarrow \overline{\Q}_p.
\]
Let \(L_p/\Q_p\) be a finite extension containing \(\iota_p(K_0)\). We use this embedding to determine the extended \(p\)-adic valuation to \(L_p\). For the rest of this section, whenever \(\nu_p\) is applied to an element of \(K_0\), we always apply the chosen embedding \(\iota_p\) which we then use to determine the extended $p$-adic valuation in \(L_p\).
\end{definition}

\begin{lemma}\label{boundthing}
Assume \(\alpha\beta\neq 0\). With notation kept the same as in Definition~\ref{def:binet-lucas}, set
\(\Gamma:=\alpha/\beta\) and \(\Lambda:=r_1/r_2\). Then for every $n\ge 0$ such that $\Lambda^n+\Gamma^{-1}\neq0$, we have
\[
\nu_p(S_n)=\nu_p(\beta)+n\nu_p(r_2)+\nu_p(\Gamma)+\nu_p\big(\Lambda^n+\Gamma^{-1}\big).
\]
\end{lemma}

\begin{proof}
Using Definition~\ref{def:binet-lucas}, let $r_1,r_2$ be the roots of $R^2-\frac{a}{b}R-\frac{c}{d}=0$. Then every term of $S_n$ can be written as
\[
S_n=\alpha r_1^n+\beta r_2^n,
\]
with $\alpha,\beta$ determined by the initial conditions and $r_1,r_2$. Factoring out $r_2^n\beta\Gamma$ gives
\[
S_n=\Gamma\beta r_2^n\Big(\Lambda^n+\Gamma^{-1}\Big),\qquad
\Gamma:=\frac{\alpha}{\beta},\quad \Lambda:=\frac{r_1}{r_2},
\]
thus
\[
\nu_p(S_n)=\nu_p(\beta)+n\nu_p(r_2)+\nu_p(\Gamma)+\nu_p\big(\Lambda^n+\Gamma^{-1}\big).
\qedhere
\]
\end{proof}

From the hypotheses in the theorem, we must have $\nu_p(a)=\nu_p(b)=\nu_p(c)=0$ for every prime $p\mid d$. Let us now consider the characteristic equation for $S_n$:
\[
R^2-\frac{a}{b}R-\frac{c}{d}=0.
\]
To make our notation simpler: write $u:=\frac{a}{b},\space v:=\frac{c}{d},\space
\Delta:=u^2+4v$, and let $r_1,r_2$ be the roots of the aforementioned characteristic polynomial:
\[
r_{1,2}=\frac{u\pm\sqrt{\Delta}}{2}.
\]
Fixing a prime $p\mid d$ and setting $\delta:=\nu_p(d)\ge1$, since $\gcd(a,d)=\gcd(b,d)=\gcd(c,d)=1$, we have $\nu_p(u)=0,\space \nu_p(v)=-\delta$.

\begin{lemma}\label{lem:distinct-roots}
Using the hypotheses \(b,d\ge 2\), \(\gcd(a,b)=\gcd(c,d)=1\), and
\(\gcd(a,d)=\gcd(b,d)=1\). We have that $\Delta$ is nonzero and \(r_1\neq r_2\).
\end{lemma}

\begin{proof}
If \(\Delta=0\), then $\frac{a^2}{b^2}=-\,\frac{4c}{d},$
so $a^2d=-4b^2c$. Thus \(b^2\mid a^2d\). Since \(\gcd(a,b)=1\) and \(\gcd(b,d)=1\), we see that \(\gcd(b,a^2d)=1\). This implies that \(b^2=1\), contradicting the assumption that \(b\ge 2\).
\end{proof}

\begin{lemma}
\label{lem:odd-p-roots}

If $p\ge3$ divides $d$, again setting $\delta=\nu_p(d)\ge1$, then
\[
\nu_p(r_1)=\nu_p(r_2)=-\frac{\delta}{2}.\qedhere
\]
\end{lemma}

\begin{proof}
Using the quadratic formula, we see that:
\[
\Delta=u^2+4v.
\]
Since this is a sum of two terms with potentially different $p$-adic valuation, we invoke the ultrametric inequality,
\[
\nu_p(\Delta)\ge \min\bigl(\nu_p(u^2),\nu_p(4v)\bigr)
=\min\bigl(2\nu_p(u),2\nu_p(2)+\nu_p(v)\bigr).
\]
We know that \(p\ge 3\), so \(\nu_p(2)=0\). Also \(\nu_p(u)=0\) and \(\nu_p(v)=-\delta\). $2\nu_p(u)=0,\space 2\nu_p(2)+\nu_p(v)=0+(-\delta)=-\delta.$ We can thus conclude that:
\[
\nu_p(\Delta)\ge \min(0,-\delta)=-\delta.
\]
Because \(-\delta<0\), the two valuations \(0\) and \(-\delta\) are unequal:
\[
\nu_p(\Delta)=\min\bigl(\nu_p(u^2),\nu_p(4v)\bigr)=-\delta.
\]

Therefore $\nu_p(\sqrt{\Delta})=\frac{1}{2}\nu_p(\Delta)=-\frac{\delta}{2}.$ Let's now compare the two terms in \(u\pm \sqrt{\Delta}\):
\[
\nu_p(u)=0,\qquad \nu_p(\sqrt{\Delta})=-\frac{\delta}{2}.
\]
Again these are unequal, since \(\delta\ge 1\). Therefore
\[
\nu_p(u\pm \sqrt{\Delta})
=\min\!\left(\nu_p(u),\nu_p(\sqrt{\Delta})\right)
=\min\!\left(0,-\frac{\delta}{2}\right)
=-\frac{\delta}{2}.
\]
Lastly, because \(p\ge 3\), we know that \(\nu_p(2)=0\), so
\[
\nu_p(r_{1,2})
=\nu_p(u\pm \sqrt{\Delta})-\nu_p(2)
=-\frac{\delta}{2}-0
=-\frac{\delta}{2}.
\]
This finally proves what we wanted:
\[
\nu_p(r_1)=\nu_p(r_2)=-\frac{\delta}{2}.
\]
\end{proof}

\begin{lemma}
\label{lem:dyadic-roots}

Let $p=2$ divide $d$ and let $\delta=\nu_2(d)\ge1$. Then we again have that
\[
\nu_2(r_1)=\nu_2(r_2)=-\frac{\delta}{2}.
\]
\end{lemma}

\begin{proof}
This proof is done using heavy casework. Start by writing $s:=u+\sqrt{\Delta}$ and $t:=u-\sqrt{\Delta}$. Note the identities:
\[
s+t=2u,\qquad s-t=2\sqrt{\Delta},\qquad st=u^2-\Delta=\left(\frac{a}{b}\right)^2-\left(\frac{a}{b}\right)^2-4\frac{c}{d}=-4v.
\]
Since $\nu_2(u)=0$ and $\nu_2(v)=-\delta$, we have $\nu_2(st)=\nu_2(-4v)=2-\delta$.
Now, we break down the proof into three cases depending on the value of $\delta$.

\emph{Case 1:  $\delta\ge3$.} Here we have $\nu_2(v)\leq -3$.   Thus
\[
\nu_2(\Delta)\geq  \min(2\nu_2(u),2\nu_2(2)+\nu_2(v))=\min(0,2-\delta)=2-\delta.
\]
We know that \(0\neq 2-\delta\) and \(2-\delta<0\), because $\delta\geq3$. From the identities above, we realize that \(\nu_2(\Delta)=2-\delta\). Thus
\[
\nu_2(r_{1,2})=\min\!\left(\nu_2(u),1-\frac{\delta}{2}\right)-1.
\]
Since \(\nu_2(u)\neq 1-\frac{\delta}{2}\), this shows that $\nu_2(r_{1,2})=1-\frac{\delta}{2}-1=-\frac{\delta}2$.

\emph{Case 2: $\delta=2$.} In this case \(\nu_2(st)=\nu_2(4v)=2-\delta=0\) and \(\nu_2(s+t)=\nu_2(2u)=1\). Due to the ultrametric inequality, if $\nu_2(s)\neq\nu_2(t)$, then
\[
\nu_2(s+t)=\min(\nu_2(s),\nu_2(t))\le0,
\]
this, however, contradicts with the fact that $\nu_2(s+t)=\nu_2(2u)=1$. Thus $\nu_2(s)=\nu_2(t)$. Set $\nu_2(s)=\nu_2(t)=\lambda$. Since $2\lambda=\nu_2(st)=0$, we have $\lambda=0$. Thus
\[
\nu_2(r_{1,2})=\lambda-\nu_2(2)=0-1=-1=-\delta/2.
\]

\emph{Case 3:  $\delta=1$.} Now $\nu_2(st)=\nu_2(4v)=2-\delta=1$, $\nu_2(s+t)=\nu_2(2u)=1$, and $\nu_2(s-t)=\nu_2(2\sqrt{\Delta})=1+\nu_2(\sqrt{\Delta})$.
If $\nu_2(s)\neq\nu_2(t)$, then
\[
\nu_2(s+t)=\min(\nu_2(s),\nu_2(t))=1.
\]
This forces $\nu_2(s)+\nu_2(t)\ge2$ which contradicts the fact that $\nu_2(st)=1$.
Thus $\nu_2(s)=\nu_2(t)=\lambda$, so $2\lambda=\nu_2(st)=1$ and $\lambda=\tfrac12$. Thus
\[
\nu_2(r_{1})=\nu_2(r_{2})=\lambda-\nu_2(2)=\tfrac12-1=-\tfrac12=-\delta/2.
\]
Since the three cases cover all possible values of $\delta$, this proves that for $p=2$, $\nu_2(r_1)=\nu_2(r_2)=-\frac{\delta}{2}$.
\end{proof}

\begin{corollary}\label{cor:alpha-unit}
Let $p\mid d$. Then $\nu_p(\Lambda)=0$.
\end{corollary}

\begin{proof}
By Lemmas~\ref{lem:odd-p-roots} and~\ref{lem:dyadic-roots}, for all primes $p$ we have $\nu_p(r_1)=\nu_p(r_2)$, hence $\nu_p(\Lambda)=\nu_p(r_1/r_2)=0$.
\end{proof}

\begin{corollary}\label{cor:r2-valuation}
Let $p\mid d$ and set $\delta:=\nu_p(d)\ge 1$. Then $\nu_p(r_2)=-\delta/2$.
\end{corollary}

\begin{proof}
If $p\ge 3$ this is Lemma~\ref{lem:odd-p-roots}, and if $p=2$ this is Lemma~\ref{lem:dyadic-roots}.
\end{proof}
\begin{lemma}\label{lem:one-coefficient-zero}
If $\alpha\beta=0$, for every prime \(p\mid d\) with \(\delta=\nu_p(d)\), we have
\[
\nu_p(S_n)=\nu_p(\gamma)-\frac{\delta}{2}n,
\]
where \(\gamma=\alpha\) when \(\beta=0\), and where \(\gamma=\beta\) in the case \(\alpha=0\). Therefore, for each such $p$ there is an explicit cutoff $k_p$ such that $\nu_p(S_n)<0$ for all $n\ge k_p$. This finally shows that $S_n\notin\Z$ for all $n\ge k_p$.
\end{lemma}

\begin{proof}
If \(\beta=0\), then \(S_n=\alpha r_1^n\). If \(\alpha=0\), then \(S_n=\beta r_2^n\).
Lemmas~\ref{lem:odd-p-roots} and~\ref{lem:dyadic-roots} both show that for each prime \(p\mid d\), $\nu_p(r_1)=\nu_p(r_2)=-\frac{\delta}{2}.$ So, in either case
\[
\nu_p(S_n)=\nu_p(\gamma)+n\left(-\frac{\delta}{2}\right)
=\nu_p(\gamma)-\frac{\delta}{2}n.
\]
Thus, once $n$ is greater than $\frac{2\nu_p(\gamma)}{\delta}$, we have $\nu_p(S_n)<0$.
\end{proof}

\begin{lemma}\label{lem:Lambda-not-rootunity}
 Assume $b,d\ge 2$, $\gcd(a,b)=\gcd(c,d)=1$, and $\gcd(b,d)=1$. Then $\Lambda=r_1/r_2$ is not a root of unity.
\end{lemma}

\begin{proof}
Proceed by contradiction and assume that $\Lambda$ is a root of unity. Since
$\Lambda\in \Q(\sqrt{\Delta})$, the field $\Q(\Lambda)$ has degree at most $2$
over $\Q$. So, if $n$ denotes the order of $\Lambda$, then $\varphi(n)\le 2$, so $n\in\{1,2,3,4,6\}.$ If $n=1$, then $\Lambda=1$, so $r_1=r_2$, which contradicts
Lemma~\ref{lem:distinct-roots}. Since $\Lambda=r_1/r_2$, we have $r_1=\Lambda r_2$. By Vieta's formulas, $r_1+r_2=\frac ab,\space r_1r_2=-\frac cd$. Thus $(\Lambda+1)r_2=\frac ab$.

If $\Lambda=-1$, then $a/b=0$, so $a=0$. But $\gcd(a,b)=1$ would then force
$b=1$, contradicting $b\ge 2$. Hence, for the remaining cases,
$\Lambda+1\ne 0$, and $r_2=\frac{a}{b(\Lambda+1)}$, we can substitute them into $r_1r_2=-c/d$. After doing so we get
\[
-\frac cd=\Lambda r_2^2
=
\Lambda\frac{a^2}{b^2(\Lambda+1)^2}.
\]
Therefore
\[
\frac cd
=
-\frac{\Lambda}{(\Lambda+1)^2}\cdot\frac{a^2}{b^2}.
\]
Since
\[
\frac{\Lambda}{(\Lambda+1)^2}
=
\frac{1}{\Lambda+\Lambda^{-1}+2},
\]
and for primitive roots of unity of orders $3,4,6$ we know that $\Lambda+\Lambda^{-1}$ is either $1,0\text{ or }-1$, it is thus clear that $\Lambda+\Lambda^{-1}+2\in\{1,2,3\}$. This shows us that
\[
\frac cd=-\frac{a^2}{t b^2}
\]
for some $t\in\{1,2,3\}$. Now we can reduce the fraction $-a^2/(t b^2)$ into its lowest terms. Since $a$ and $b$ are relatively prime, no prime divisor of $b$ divides $a$. Therefore no prime power coming from $b^2$ can cancel with the numerator $a^2$, even if some prime divisor of $b$ also divides $t$. Thus the reduced denominator is still divisible by $b^2$. But $c/d$ is already in lowest terms, so $d$ must be divisible by $b^2$. This contradicts $\gcd(b,d)=1$ and $b\ge 2$. Therefore $\Lambda$ is not a root of unity.
\end{proof}

\begin{lemma}\label{lem:Gamma-split}
Fix a prime $p\mid d$.

\smallskip
\noindent (i) If $\nu_p(\Gamma^{-1})\neq 0$, then for every $n\ge 1$ we have
\[
\nu_p(\Lambda^n+\Gamma^{-1})
\le 0.
\]

\smallskip
\noindent (ii) If $\nu_p(\Gamma^{-1})=0$, then $\Lambda$ and $-\Gamma^{-1}$ have $p$-adic valuation $0$.
\end{lemma}

\begin{proof}
We know from Corollary~\ref{cor:alpha-unit} that $\nu_p(\Lambda)=0$, hence $\nu_p(\Lambda^n)=0$ for all $n\ge 1$. If $\nu_p(\Gamma^{-1})\neq 0$, then $\nu_p(\Lambda^n)\neq \nu_p(\Gamma^{-1})$, so the ultrametric inequality gives
\[
\nu_p(\Lambda^n+\Gamma^{-1})=\min(\nu_p(\Lambda^n),\nu_p(\Gamma^{-1}))
=\min(0,\nu_p(\Gamma^{-1}))\le 0,
\]
proving (i). If $\nu_p(\Gamma^{-1})=0$, (ii) is immediate from $\nu_p(\Lambda)=0$.
\end{proof}
\begin{lemma}\label{lem:chim-log-cor}
Start by letting \(\overline{\Q}\) be an algebraic closure of \(\Q\), now fix some embedding
\(\overline{\Q}\hookrightarrow \overline{\Q}_p\).
Define \(\alpha_1,\alpha_2\in \overline{\Q}\), also let them satisfy the property that
\[
\nu_p(\alpha_1)=\nu_p(\alpha_2)=0.
\]
Furthermore, we assume that \(\alpha_1\) and \(\alpha_2\) are multiplicatively independent. Finally, we fix a positive integer \(B_0\). Then there must exist some explicit constants \(A,B>0\), depending only on \(\alpha_1,\alpha_2,B_0\), and \(p\), such that for every integer \(m\ge 1\) with $\alpha_1^m-\alpha_2^{B_0}\neq 0$, we know that
\[
\nu_p(\alpha_1^m-\alpha_2^{B_0})\le A\log m+B.
\]
\end{lemma}

\begin{proof}
First let
\[
D=[\Q(\alpha_1,\alpha_2):\Q],
\]
next choose real numbers \(A_1,A_2>1\) such that
\[
\log A_i\ge \max\!\left\{h(\alpha_i),\frac{\log p}{D}\right\}
\qquad (i=1,2),
\]
where \(h(\alpha_i)\) symbolizes the absolute logarithmic Weil height of some number $\alpha_i$. Remember that the only things these choices depend on are \(\alpha_1,\alpha_2\), and \(p\). Now we can apply Chim's theorem \cite[Theorem~2.1]{Chim2025} to the quantity
\[
\alpha_1^m-\alpha_2^{B_0},
\]
with exponents \(b_1=m\) and \(b_2=B_0\). Since \(\alpha_1,\alpha_2\) are algebraic,
\(p\)-adic units, and multiplicatively independent, the hypotheses of the theorem
are satisfied. Since these hypotheses are satisfied, we are given explicit constants \(C_0,C_1>0\), depending only on
\(\alpha_1,\alpha_2,B_0\), and \(p\), such that the inequality always holds:
\[
\nu_p(\alpha_1^m-\alpha_2^{B_0})
\le C_0 \log\!\bigl(mD\log A_2+B_0D\log A_1\bigr)+C_1
\]
for every integer \(m\ge 1\) with \(\alpha_1^m-\alpha_2^{B_0}\neq 0\). Since \(B_0\) always remains constant, there exists a constant, called \(C_2>0\), depending only on
\(\alpha_1,\alpha_2,B_0\), and \(p\), where:
\[
mD\log A_2+B_0D\log A_1 \le C_2 m
\qquad (m\ge 1).
\]
Due to the above inequality, we can use a little bit of grouping to calculate explicit values for $A$ and $B$
\[
\nu_p(\alpha_1^m-\alpha_2^{B_0})
\le C_0\log(C_2m)+C_1
= C_0\log m + \bigl(C_0\log C_2+C_1\bigr).
\]
Specifically, we find that the claim will hold with
\[
A:=C_0,
\qquad
B:=C_0\log C_2+C_1.\qedhere
\]
\end{proof}

\begin{lemma}\label{lem:unit-minus-rootofunity}
Let $\Q_p$ be the field of $p$-adic numbers, $L/\Q_p$ be a finite field extension, with $\nu_p$ extended as in Definition~\ref{def:vp-extension}, and $\mathcal O_L := \{x \in L : \nu_p(x) \ge 0\}$ denote the ring of integers of $L$. For some integer $c\ge 1$, we write
\[
\mathfrak P_{\ge c}:=\{x\in L:\nu_p(x)\ge c\},\qquad
1+\mathfrak P_{\ge c}:=\{x\in \mathcal O_L^\times:x\equiv 1\pmod{\mathfrak P_{\ge c}}\}.
\]
After doing that, let $\theta \in L^\times$ be a $p$-adic unit that is not a root of unity. Then there must exist constants $A,B$, depending only on $L$ and $\theta$, such that for every integer $m \ge 1$ with $\theta^m \neq 1$,
\[
\nu_p(\theta^m - 1) \le A \log m + B.
\]
\end{lemma}

\begin{proof}
The general idea of this proof is to use the property that in the $p$-adic number field, $\nu_p(\log(1+x))=\nu_p(x)$. Start by choosing an integer $c\ge1$ large enough so that the series
\[
\log(1+x)=\sum_{n\ge1}\frac{(-1)^{n+1}x^n}{n}
\]
converges for all $x\in\mathfrak P_{\ge c}$. For instance, one can take any integer $c>1/(p-1)$. This is because for such $c$, the first term in the logarithm, namely $x$, has a smaller $p$-adic valuation than all later terms, and therefore, from the ultrametric inequality:
\[
\nu_p(\log(1+x))=\nu_p(x)
\]
for all $x\in\mathfrak P_{\ge c}$. We can also realize that since $\log(1+x)=0$ implies that $x=0$, the logarithm is injective on $1+\mathfrak P_{\ge c}$. Furthermore, we know the $p$-adic logarithm also satisfies $\log(uv)=\log(u)+\log(v)$ whenever $u,v\in 1+\mathfrak P_{\ge c}$, thus we have the injective group homomorphism
\[
\log:1+\mathfrak P_{\ge c}\longrightarrow \mathfrak P_{\ge c}.
\]
Because we know the size of $\mathcal O_L^\times/(1+\mathfrak P_{\ge c})$, we can define $M$ to be a positive integer divisible by the order of said group. Then $\theta^M\in 1+\mathfrak P_{\ge c}$.
Since $\theta$ is not a root of unity, $\theta^M$ is not a root of unity either, and since $\log$ is injective on $1+\mathfrak P_{\ge c}$, we have $\log(\theta^M)\neq 0.$ Next we take advantage of the division algorithm and write
\[
m=Mq+r, \qquad 0\le r<M.
\]
If $q=0$, then $m=r$, so $1\le m<M$. Thus there are only finitely many
such integers $m$, and for each of them the value $\nu_p(\theta^m-1)$ is
finite because $\theta^m\neq 1$. This just means that for each of these finitely many initial cases, we can absorb each of their constant $p$-adic valuations into the final bound. Having dealt with the case where $q=0$, we can deal with the cases where $q\ge 1$. First deduce that $\theta^m-1=\theta^r\bigl((\theta^M)^q-\theta^{-r}\bigr)$, so since $\theta$ is a $p$-adic unit,
\[
\nu_p(\theta^m-1)=\nu_p\bigl((\theta^M)^q-\theta^{-r}\bigr).
\]
This allows us to solve the problem by dealing with two cases.

\medskip
\noindent\textbf{Case 1: $\theta^{-r}\notin 1+\mathfrak P_{\ge c}$.}
If $\theta^{-r}$ is indeed not one away from a high power of $\mathfrak P$, we can simply take advantage of the ultrametric inequality:
\[
(\theta^M)^q-1\in \mathfrak P_{\ge c},
\]
while $1-\theta^{-r}\notin \mathfrak P_{\ge c}$. Therefore
\[
\nu_p\big((\theta^M)^q-1\big)\ge c,
\qquad
\nu_p(1-\theta^{-r})<c.
\]
The two summands $(\theta^M)^q-1 \text{ and }1-\theta^{-r}$ have unequal $p$-adic valuations. So using the ultrametric equality:
\[
\nu_p\!\left((\theta^M)^q-\theta^{-r}\right)
=
\nu_p(1-\theta^{-r}).
\]
So this case is bounded independently of $q$, consequently making it independent of $m$.

\medskip
\noindent\textbf{Case 2: $\theta^{-r}\in 1+\mathfrak P_{\ge c}$.}
This immediately implies that $\theta^r\in 1+\mathfrak P_{\ge c}$, and since $(\theta^M)^q-\theta^{-r}=\theta^{-r}(\theta^{Mq+r}-1)$, with $\nu_p(\theta^{-r})=0$, we get
\[
\nu_p\bigl((\theta^M)^q-\theta^{-r}\bigr)=\nu_p(\theta^{Mq+r}-1).
\]
Additionally, since $(\theta^M)^q\in 1+\mathfrak P_{\ge c}$ and $\theta^r\in 1+\mathfrak P_{\ge c}$, we have $\theta^{Mq+r}=(\theta^M)^q\theta^r\in 1+\mathfrak P_{\ge c}$. Using that $\log$ is a homomorphism on $1+\mathfrak P_{\ge c}$, we obtain
$\nu_p(\theta^{Mq+r}-1)=\nu_p\bigl(\log(\theta^{Mq+r})\bigr).$ Now, we can neatly write $\nu_p(\theta^{Mq+r}-1)$ as a nice sum of $p$-adic valuations for a few logarithms, namely: $q\log(\theta^M)+\log(\theta^r)$. Next, we find an expression for $\log(\theta^r)$ in terms of $\log(\theta^M)$:
\[
M\log(\theta^r)=\log((\theta^r)^M)=\log(\theta^{Mr})=r\log(\theta^M),
\] 
so $\log(\theta^r)=\frac{r}{M}\log(\theta^M).$ Therefore
\[
\log(\theta^{Mq+r})
=
q\log(\theta^M)+\log(\theta^r)
=
\left(q+\frac{r}{M}\right)\log(\theta^M)
=
\frac{Mq+r}{M}\log(\theta^M).
\]
Realize that because $m=Mq+r$: $\log(\theta^{Mq+r})=\frac{m}{M}\log(\theta^M)$. Thus
\[
\nu_p\bigl((\theta^M)^q-\theta^{-r}\bigr)
=
\nu_p(m)-\nu_p(M)+\nu_p(\log(\theta^M)).
\]
Using a simple upper bound for $p$-adic valuation, clearly $\nu_p(m)\le \log_p m=\frac{\log m}{\log p}$, thus
\[
\nu_p\bigl((\theta^M)^q-\theta^{-r}\bigr)
\le
\frac{\log m}{\log p}+C',
\]
where $C':=-\nu_p(M)+\nu_p(\log(\theta^M))$. Since the theorem requires a bound for all positive integers $m$, not just for a certain residue class modulo $M$, we must handle each residue $r\in\{0,\ldots,M-1\}$. Luckily, for each residue $r$, the preceding argument gives $A_r,B_r$ such that
\[
\nu_p(\theta^m-1)\le A_r\log m+B_r
\]
for every integer $m\ge 1$ with $m\equiv r\pmod M$ and $\theta^m\ne 1$. Because there are only finitely many residues all we must do is set $A:=\max_{0\le r<M}A_r, \quad B:=\max_{0\le r<M}B_r.$ Then
\[
\nu_p(\theta^m-1)\le A\log m+B
\]
for every integer $m\ge 1$ with $\theta^m\ne 1$.
\end{proof}

\begin{lemma}\label{lem:log-upper}
Fix a prime \(p\mid d\), assume \(\alpha\beta\neq 0\) and \(\nu_p(\Gamma^{-1})=0\). Then there must exist explicit constants $C_2,C_3$ (depending only on $a,b,c,d,S_0,S_1$ and $p$) such that for all $n\ge 2$ with
\(\Lambda^n+\Gamma^{-1}\ne 0\),
\[
\nu_p(\Lambda^n+\Gamma^{-1})\le C_2\log n + C_3.
\]
\end{lemma}

\begin{proof}
We immediately simplify our notation by defining:
\[
\alpha_1:=\Lambda,\qquad \alpha_2:=-\Gamma^{-1}.
\]
By Corollary~\ref{cor:alpha-unit} and the hypothesis that \(\nu_p(\Gamma^{-1})=0\), it clearly follows that: $\nu_p(\alpha_1)=\nu_p(\alpha_2)=0$. Furthermore, we know that, \(\alpha_1\) and \(\alpha_2\) are algebraic over \(\Q\), since
\[
\alpha_1=\Lambda=\frac{r_1}{r_2},\qquad \alpha_2=-\Gamma^{-1}=-\frac{\beta}{\alpha},
\]
and \(r_1,r_2,\alpha,\beta\) are algebraic. First, let us assume that \(\alpha_1\) and \(\alpha_2\) are multiplicatively independent. Then Lemma~\ref{lem:chim-log-cor} applies to
\[
\alpha_1^n-\alpha_2=\Lambda^n+\Gamma^{-1},
\]
since the exponents \(n\) and \(1\) are positive integers for \(n\ge 2\). Thus, we are done because there exist explicit constants \(C_2,C_3\) such that
\[
\nu_p(\Lambda^n+\Gamma^{-1})\le C_2\log n + C_3
\]
for all \(n\ge 2\). Since we are done with the simpler independent case, we now assume \(\alpha_1\) and \(\alpha_2\) are multiplicatively dependent. From the definition of dependence there exist coprime integers \(r,s\) with \(s>0\) such that $\alpha_1^r=\alpha_2^s$. To make the $\alpha_i$'s be easier to deal with, we let \(L/\Q_p\) be a finite extension containing an \(s\)-th root \(\eta\) of \(\alpha_1\) and an \(s\)-th root of unity \(\xi\) such that $\alpha_1=\eta^s,\space \space\alpha_2=\xi\,\eta^r$. Then
\[
\alpha_1^n-\alpha_2=\eta^{sn}-\xi\eta^r=\eta^r(\eta^{sn-r}-\xi).
\]
Since \(\nu_p(\eta^s)=\nu_p(\alpha_1)=0\), we have \(\nu_p(\eta)=0\), so \(\eta\) is a \(p\)-adic unit. Also, \(\eta\) is not a root of unity, for otherwise \(\alpha_1=\eta^s\) would be a root of unity, which contradicts Lemma~\ref{lem:Lambda-not-rootunity}. Set $m_n:=sn-r$. Then
\[
\nu_p(\Lambda^n+\Gamma^{-1})
=\nu_p(\alpha_1^n-\alpha_2)
=\nu_p(\eta^{m_n}-\xi).
\]
Because \(s>0\), there are only finitely many integers \(n\ge 2\) such that \(m_n\le 0\). Under the assumption \(\Lambda^n+\Gamma^{-1}\ne 0\), the quantities $\nu_p(\Lambda^n+\Gamma^{-1})$ for those finitely many \(n\) are finite; let \(C_0\) be their maximum. Now suppose \(m_n\ge 1\). Let \(q\) be the order of the root of unity \(\xi\). Since \(\eta\) is not a root of unity, \(\eta^q\) is not a root of unity either. Hence $(\eta^q)^{m_n}\ne 1$ for every \(m_n\ge 1\). Also, \(\eta^q\) is a \(p\)-adic unit. After replacing \(L\) by a finite extension if necessary, assume that \(L\) contains all \(q\)-th roots of unity. We compare \(\eta^{m_n}-\xi\) with \(\eta^{qm_n}-1\). Since \(\xi^q=1\), we have
\[
\eta^{qm_n}-1
=
\prod_{\zeta^q=1}(\eta^{m_n}-\zeta)
=
(\eta^{m_n}-\xi)
\prod_{\substack{\zeta^q=1\\ \zeta\neq \xi}}
(\eta^{m_n}-\zeta).
\]
All factors in this product are differences of \(p\)-adic units, so they have nonnegative \(p\)-adic valuation. Therefore
\[
\nu_p(\eta^{m_n}-\xi)
\le
\nu_p(\eta^{qm_n}-1)
=
\nu_p\bigl((\eta^q)^{m_n}-1\bigr).
\]
Applying Lemma~\ref{lem:unit-minus-rootofunity} to \(\theta=\eta^q\), we obtain constants \(A_1,B_1\) such that

\[
\nu_p(\Lambda^n+\Gamma^{-1})\le A_1\log m_n+B_1
\qquad (m_n\ge 1).
\]
Since for all $n$ greater than or equal to one $m_n=sn-r\le sn+|r|\le (s+|r|)n$, we obtain
\[
\nu_p(\Lambda^n+\Gamma^{-1})
\le A_1\log n+\bigl(B_1+A_1\log(s+|r|)\bigr).
\]
Combining this with the finitely many indices \(n\ge 2\) for which \(m_n\le 0\), we obtain constants \(C_2,C_3\) such that for all \(n\ge 2\) with $\Lambda^n+\Gamma^{-1}\ne 0$, one has
\[
\nu_p(\Lambda^n+\Gamma^{-1})\le C_2\log n+C_3.\qedhere
\]
\end{proof}

\begin{lemma}\label{lem:linear-lb}
Assume \(\alpha\beta\neq 0\). Fix a prime $p\mid d$ and set $\delta:=\nu_p(d)\ge 1$.
If $\nu_p(S_n)\ge 0$ and $\Lambda^n+\Gamma^{-1}\ne 0$, then
\begin{equation}\label{eq:linear-lb}
\nu_p(\Lambda^n+\Gamma^{-1})\ge \frac{\delta}{2}n-\nu_p(\beta)-\nu_p(\Gamma).
\end{equation}
\end{lemma}

\begin{proof}
If one rearranges the result from Lemma~\ref{boundthing}, one gets
\[
\nu_p\big(\Lambda^n+\Gamma^{-1}\big)=\nu_p(S_n)-\nu_p(\beta)-n\nu_p(r_2)-\nu_p(\Gamma).
\]
By Corollary~\ref{cor:r2-valuation} we know that $\nu_p(r_2)=-\delta/2$ for $p\mid d$.
If $\nu_p(S_n)\ge 0$ then substituting $\nu_p(r_2)=-\delta/2$ yields \eqref{eq:linear-lb}.
\end{proof}
\begin{proof}[Proof of Theorem~\ref{thm:main-cutoff}]
Fix a prime $p\mid d$, and set $\delta:=\nu_p(d)\ge 1.$ First suppose $\alpha\beta=0$. If $\beta=0$, then $S_n=\alpha r_1^n$; if
$\alpha=0$, then $S_n=\beta r_2^n$. In both cases, Corollary~\ref{cor:r2-valuation} gives us, $\nu_p(S_n)=\nu_p(\gamma)-\frac{\delta}{2}n,$ where $\gamma=\alpha$ in the first case and $\gamma=\beta$ in the second. All we need to do then is just rearrange the bound for $n$, after doing this we see that $n>\frac{2\nu_p(\gamma)}{\delta}.$ Therefore, in this case, we may take
\[
k_p
:=
\max\left(
1,\left\lfloor \frac{2\nu_p(\gamma)}{\delta}\right\rfloor+1
\right).
\]
Then $\nu_p(S_n)<0$ for every $n\ge k_p$.

Since we dealt with the case where $\alpha\beta=0$, for the rest of the proof, we can assume $\alpha\beta\ne 0$. First we recall that $S_n=\Gamma\beta r_2^n(\Lambda^n+\Gamma^{-1})$. One thing we can quickly deduce is that at most one integer $n\ge 0$ satisfies $\Lambda^n+\Gamma^{-1}=0$, because if two distinct integers $m,n$ satisfied this equation, then $\Lambda^m=\Lambda^n,$ so $\Lambda^{m-n}=1$, which would contradict Lemma~\ref{lem:Lambda-not-rootunity}. If such an integer exists, call it $n_0$; otherwise set $n_0:=0$. In the final cutoff below we will require $n>n_0$, so that $\Lambda^n+\Gamma^{-1}\ne 0$. We can start the case where $\nu_p(\Gamma^{-1})\ne 0.$ By Lemma~\ref{lem:Gamma-split}, $\nu_p(\Lambda^n+\Gamma^{-1})\le 0$ for every $n\ge 1$. If $\nu_p(S_n)\ge 0$, then Lemma~\ref{lem:linear-lb} gives
\[
0\ge \frac{\delta}{2}n-\nu_p(\beta)-\nu_p(\Gamma).
\]
We can just re-write this bound in terms of $n$, which gets us:
\[
n\le \frac{2(\nu_p(\beta)+\nu_p(\Gamma))}{\delta}.
\]
Therefore, if we use the bound:
\[
k_p
:=
\max\left(
1,n_0+1,
\left\lfloor
\frac{2(\nu_p(\beta)+\nu_p(\Gamma))}{\delta}
\right\rfloor+1
\right).
\]
Then $\nu_p(S_n)<0$ for all $n\ge k_p$.

Now, we can treat the case where $\nu_p(\Gamma^{-1})=0.$ By Lemma~\ref{lem:log-upper}, there exist constants $C_2,C_3$ such that, for all $n\ge 2$ with $\Lambda^n+\Gamma^{-1}\ne 0$,
\[
\nu_p(\Lambda^n+\Gamma^{-1})\le C_2\log n+C_3.
\]
If we want $\nu_p(S_n)\ge 0$, then Lemma~\ref{lem:linear-lb} implies that the inequality must also hold:
\[
\frac{\delta}{2}n-\nu_p(\beta)-\nu_p(\Gamma)
\le
\nu_p(\Lambda^n+\Gamma^{-1}).
\]
Combining the two inequalities gives
\[
\frac{\delta}{2}n
\le
C_2\log n+\nu_p(\beta)+\nu_p(\Gamma)+C_3.
\]
Now, all we must do is find a way to rearrange and solve for n. We can start by setting $L:=\nu_p(\beta)+\nu_p(\Gamma)+C_3$ and $L_+:=\max(L,0).$ Then
\[
\frac{\delta}{2}n\le C_2\log n+L_+.
\]
For $n\ge 16$, we have $\log n\le \sqrt n$, so for $n\ge16$ the inequality holds:
\[
\frac{\delta}{2}n\le C_2\sqrt n+L_+.
\]
This becomes $\frac{\delta}{2}(\sqrt n)^2-C_2(\sqrt n)-L_+\le 0.$ So we can use the quadratic formula to solve the inequality for $\sqrt n$:
\[
\sqrt n\le
\frac{C_2+\sqrt{C_2^2+2\delta L_+}}{\delta}.
\]
Now we can just square the solution to solve for the upper bound on $n$

\begin{equation}
\label{eq:upper_bound_final}
n\leq
\left(
\frac{C_2+\sqrt{C_2^2+2\delta L_+}}{\delta}
\right)^2.
\end{equation}

Write $R_p:=\left(\dfrac{C_2+\sqrt{C_2^2+2\delta L_+}}{\delta}\right)^2$ for the right-hand side of \eqref{eq:upper_bound_final}. Then no $n\ge 16$ that does not satisfy \eqref{eq:upper_bound_final} can have $\nu_p(S_n)\ge 0$. Thus, in this case, we may take
\[
k_p
:=
\max\left(
16,n_0+1,\lfloor R_p\rfloor+1
\right).
\]
Then $\nu_p(S_n)<0$ for every $n\ge k_p$.

To summarize, for each prime $p\mid d$, the preceding cases construct a cutoff $k_p$ such that $\nu_p(S_n)<0$ for all $n\ge k_p$. Let
\[
k:=\max_{p\mid d} k_p.
\]
Then, for every $n\ge k$, at least one prime $p\mid d$ satisfies $\nu_p(S_n)<0.$ Therefore $S_n\notin\mathbb Z$ for all $n\ge k$. This proves the theorem.
\end{proof}

\section*{Sequences from the OEIS}

The OEIS sequences discussed explicitly in this paper are the sequence of powers of $3$, \seqnum{A000244}, which appears in Remark~\ref{rem:normalizations}, and the sequence \seqnum{A133594}, which appears in Example~\ref{ex:oeis-a133594} as an auxiliary recurrence associated to a fractional-coefficient recurrence.

\section*{Disclosure of AI use}
While the two theorems studied in this paper, the overall strategy, and the key ideas are the author's own, the author used the large language models ChatGPT and Claude to check the manuscript for mathematical and computational errors, to verify the numerical examples and OEIS data, to improve the English exposition, and to format the paper to the journal's style guide. These tools were also used to help make several of the author's proof ideas rigorous and to explain background material with which the author was less familiar. Specifically, they aided in the matrix argument in the proof of Lemma~\ref{lem:special-coprime-clean}, and with a number of the field-theoretic, $p$-adic, group-theoretic, and polynomial-root steps in Section~\ref{sec:cutoff}. The author independently checked every AI-assisted argument, wrote the text of the paper, and takes full responsibility for the correctness and originality of all results and proofs.

\section*{Acknowledgments}
The author expresses his deepest gratitude to his mentor Dr.\ Rowland, who provided the question that led the author into this area of research and whose careful guidance during the paper-writing process made this work possible. The author also thanks Hamza Virk for many useful conversations throughout the research.

\end{document}